\documentclass{amsart}
\usepackage{amsopn}
\usepackage{amssymb, amscd}

\newcommand{\nc}{\newcommand}

\nc{\fg}{\mathfrak{f} } \nc{\vg}{\mathfrak{v} } \nc{\wg}{\mathfrak{w} }
\nc{\zg}{\mathfrak{z} } \nc{\ngo}{\mathfrak{n} } \nc{\kg}{\mathfrak{k} }
\nc{\mg}{\mathfrak{m} } \nc{\bg}{\mathfrak{b} } \nc{\ggo}{\mathfrak{g} }
\nc{\ggob}{\overline{\mathfrak{g}} } \nc{\sog}{\mathfrak{so} }
\nc{\sug}{\mathfrak{su} } \nc{\spg}{\mathfrak{sp} } \nc{\slg}{\mathfrak{sl} }
\nc{\glg}{\mathfrak{gl} } \nc{\cg}{\mathfrak{c} } \nc{\rg}{\mathfrak{r} }
\nc{\hg}{\mathfrak{h} } \nc{\tg}{\mathfrak{t} } \nc{\ug}{\mathfrak{u} }
\nc{\dg}{\mathfrak{d} } \nc{\ag}{\mathfrak{a} } \nc{\pg}{\mathfrak{p} }
\nc{\sg}{\mathfrak{s} } \nc{\affg}{\mathfrak{aff} }

\nc{\pca}{\mathcal{P}} \nc{\nca}{\mathcal{N}} \nc{\lca}{\mathcal{L}}
\nc{\oca}{\mathcal{O}} \nc{\mca}{\mathcal{M}} \nc{\tca}{\mathcal{T}}
\nc{\aca}{\mathcal{A}} \nc{\cca}{\mathcal{C}} \nc{\gca}{\mathcal{G}}
\nc{\sca}{\mathcal{S}} \nc{\hca}{\mathcal{H}} \nc{\bca}{\mathcal{B}}
\nc{\dca}{\mathcal{D}} \nc{\val}{\operatorname{val}}

\nc{\vp}{\varphi} \nc{\ddt}{\tfrac{{\rm d}}{{\rm d}t}}
\nc{\dpar}{\tfrac{\partial}{\partial t}} \nc{\im}{\mathtt{i}}

\nc{\SO}{\mathrm{SO}} \nc{\Spe}{\mathrm{Sp}} \nc{\Sl}{\mathrm{SL}}
\nc{\SU}{\mathrm{SU}} \nc{\Or}{\mathrm{O}} \nc{\U}{\mathrm{U}} \nc{\Gl}{\mathrm{GL}}
\nc{\Se}{\mathrm{S}} \nc{\Cl}{\mathrm{Cl}} \nc{\Spein}{\mathrm{Spin}}
\nc{\Pin}{\mathrm{Pin}} \nc{\G}{\mathrm{GL}_n(\RR)} \nc{\g}{\mathfrak{gl}_n(\RR)}

\nc{\RR}{{\Bbb R}} \nc{\HH}{{\Bbb H}} \nc{\CC}{{\Bbb C}} \nc{\ZZ}{{\Bbb Z}}
\nc{\FF}{{\Bbb F}} \nc{\NN}{{\Bbb N}} \nc{\QQ}{{\Bbb Q}} \nc{\PP}{{\Bbb P}}

\nc{\vs}{\vspace{.2cm}} \nc{\vsp}{\vspace{1cm}} \nc{\ip}{\langle\cdot,\cdot\rangle}
\nc{\ipp}{(\cdot,\cdot)} \nc{\la}{\langle} \nc{\ra}{\rangle} \nc{\unm}{\tfrac{1}{2}}
\nc{\unc}{\tfrac{1}{4}} \nc{\und}{\tfrac{1}{16}} \nc{\no}{\vs\noindent}
\nc{\lam}{\Lambda^2(\RR^n)^*\otimes\RR^n} \nc{\tangz}{{\rm T}^{\rm Zar}}
\nc{\nor}{{\sf n}}  \nc{\mum}{/\!\!/} \nc{\kir}{/\!\!/\!\!/}
\nc{\Ri}{\tfrac{4\Ric_{\mu}}{||\mu||^2}} \nc{\ds}{\displaystyle}
\nc{\ben}{\begin{enumerate}} \nc{\een}{\end{enumerate}} \nc{\f}{\frac}
\nc{\lb}{[\cdot,\cdot]} \nc{\isn}{\tfrac{1}{||v||^2}}

\nc{\Hess}{\operatorname{Hess}} \nc{\ad}{\operatorname{ad}}
\nc{\Ad}{\operatorname{Ad}} \nc{\rank}{\operatorname{rank}}
\nc{\Irr}{\operatorname{Irr}} \nc{\End}{\operatorname{End}}
\nc{\Aut}{\operatorname{Aut}} \nc{\Inn}{\operatorname{Inn}}
\nc{\Der}{\operatorname{Der}} \nc{\Ker}{\operatorname{Ker}}
\nc{\Iso}{\operatorname{I}} \nc{\Diff}{\operatorname{Diff}}
\nc{\Lie}{\operatorname{L}} \nc{\tr}{\operatorname{tr}} \nc{\dif}{\operatorname{d}}
\nc{\sen}{\operatorname{sen}} \nc{\modu}{\operatorname{mod}}
\nc{\Ric}{\operatorname{R}} \nc{\Ricci}{\operatorname{Ric}}
\nc{\sym}{\operatorname{sym}} \nc{\symac}{\operatorname{sym^{ac}}}
\nc{\symc}{\operatorname{sym^{c}}} \nc{\scalar}{\operatorname{sc}}
\nc{\grad}{\operatorname{grad}} \nc{\ricci}{\operatorname{Ric}}
\nc{\Nor}{\operatorname{Norm}} \nc{\riccic}{\operatorname{ric^{c}}}
\nc{\riccig}{\operatorname{ric^{\gamma}}} \nc{\Rin}{\operatorname{M}}
\nc{\Le}{\operatorname{L}} \nc{\tang}{\operatorname{T}}
\nc{\level}{\operatorname{level}} \nc{\rad}{\operatorname{r}}
\nc{\abel}{\operatorname{ab}} \nc{\CH}{\operatorname{CH}}
\nc{\mcc}{\operatorname{mcc}} \nc{\Adj}{\operatorname{Adj}}
\nc{\Order}{\operatorname{O}}

\theoremstyle{plain}
\newtheorem{theorem}{Theorem}[section]
\newtheorem{proposition}[theorem]{Proposition}
\newtheorem{corollary}[theorem]{Corollary}
\newtheorem{lemma}[theorem]{Lemma}

\theoremstyle{definition}
\newtheorem{definition}[theorem]{Definition}

\theoremstyle{remark}
\newtheorem{remark}[theorem]{Remark}

\newtheorem{example}[theorem]{Example}

\title{Ricci soliton solvmanifolds}

\author{Jorge Lauret}

\address{FaMAF and CIEM, Universidad Nacional de C\'ordoba, C\'ordoba, Argentina}
\email{lauret@famaf.unc.edu.ar}

\thanks{This research was partially supported by grants from CONICET (Argentina)
and SeCyT (Universidad Nacional de C\'ordoba)}

\begin{document}

\maketitle

\begin{abstract}
All known examples of nontrivial homogeneous Ricci solitons are left-invariant
metrics on simply connected solvable Lie groups whose Ricci operator is a multiple
of the identity modulo derivations (called {\it solsolitons}, and {\it nilsolitons}
in the nilpotent case).  The tools from geometric invariant theory used to study
Einstein solvmanifolds, turned out to be useful in the study of solsolitons as well.
We prove that, up to isometry, any solsoliton can be obtained via a very simple
construction from a nilsoliton $N$ together with any abelian Lie algebra of
symmetric derivations of its metric Lie algebra $(\ngo,\ip)$. The following
uniqueness result is also obtained: a given solvable Lie group can admit at most one
solsoliton up to isometry and scaling.  As an application, solsolitons of dimension
$\leq 4$ are classified.
\end{abstract}

\section{Introduction}\label{intro}

It has recently been proved by Lott in \cite{Ltt07} that if $g(t)$ is a Ricci flow
solution on a $3$-dimensional compact manifold $M$, with sectional curvatures that
are $\Order(t^{-1})$ and diameter that is $\Order(\sqrt{t})$, then the pullback
Ricci flow solution on the simply connected cover $\tilde{M}$ approaches a
homogeneous expanding Ricci soliton.  Among many others, this is certainly a good
motivation to study Ricci solitons in the homogeneous case.  A natural question we
are particularly interested in is how much stronger is, for homogeneous metrics, the
Einstein condition compared with the condition of being a Ricci soliton.

From results due to Ivey, Naber, Perelman and Petersen-Wylie, it follows that any
nontrivial homogeneous Ricci soliton must be noncompact, expanding and non-gradient
(see Section \ref{hrs}). Up to now, all known examples are isometric to a
left-invariant metric $g$ on a simply connected Lie group $G$, which when identified
with an inner product on the Lie algebra $\ggo$ of $G$ satisfies
\begin{equation}\label{intro1}
\Ricci(g)=cI+D, \qquad\mbox{for some}\quad c\in\RR,\quad D\in\Der(\ggo),
\end{equation}
where $\Ricci(g)$ is the Ricci operator of $g$.  On the other hand the converse is
true: any left invariant metric which satisfies (\ref{intro1}) is automatically a
Ricci soliton.  For $G$ nilpotent, these metrics are called {\it nilsolitons} and
have been extensively studied in the last decade, mainly because of the strong
connection with Einstein solvmanifolds (see the survey \cite{cruzchica}).  Examples
with $G$ solvable but non-nilpotent have explicitly appeared in \cite{BrdDnl}
($\dim{G}=3$) and \cite{IsnJckLu} ($\dim{G}=4$).

The aim of this paper is to study the structure of solvable Lie groups admitting a
left invariant metric for which (\ref{intro1}) holds; these metrics will be called
{\it solsolitons} from now on.  The tools from geometric invariant theory used in
\cite{standard} to prove that any Einstein solvmanifold is standard (see Section
\ref{Tb}), turned out to be useful in the study of solsolitons as well.

It is really easy to get examples of solsolitons from only a nilsoliton and an
abelian Lie algebra of symmetric derivations of its metric Lie algebra.  Our main
result asserts that any solsoliton can actually be obtained (up to isometry) by such
a simple construction (see Section \ref{structure}).  In particular, any solsoliton
is standard, and if not Einstein, it admits a one-dimensional extension which is an
Einstein solvmanifold, just as for nilsolitons.  We are therefore showing that most
of the structural results proved for Einstein solvmanifolds in \cite{Hbr,standard}
are still valid for solsolitons.  We also obtain a uniqueness result that
generalizes the known results for Einstein solvmanifolds and nilsolitons: among all
left invariant metrics on a given solvable Lie group, there is at most one
solsoliton up to isometry and scaling (see Section \ref{uniq}).  All this is used in
Section \ref{exa} to classify solvable Lie groups admitting solsolitons in dimension
$\leq 4$.

\vs \noindent {\it Acknowledgements.}  Part of this research was performed while
attending the Borel trimester `Ricci curvature and Ricci flow'. I am very grateful
to Institut Poincar\'e for their hospitality and to the organizers for inviting me.
I would also like to thank Luca Di Cerbo and Peng Lu for fruitful discussions on the
topic of this paper.

\section{Homogeneous Ricci solitons}\label{hrs}

A complete Riemannian metric $g$ on a differentiable manifold $M$ is said to be a
{\it Ricci soliton} if its Ricci tensor satisfies
\begin{equation}\label{defrs}
\ricci(g)=cg+\Lie_Xg, \qquad\mbox{for some}\quad c\in\RR, \quad
X\in\chi(M)\;\mbox{complete},
\end{equation}
where $\chi(M)$ denotes the space of all differentiable vector fields on $M$ and
$\Lie_X$ the Lie derivative (see \cite{libro} for further information on Ricci
solitons). Recall that if $\theta_t$ is the one-parameter group associated to $X$
then $\Lie_Xg=\ddt|_0\theta_t^*g$, and hence the Ricci soliton condition may be
rephrased as follows: $\ricci(g)$ is tangent at $g$ to the space of all metrics
which are {\it homothetic} (i.e. isometric up to scaling) to $g$.  If in addition
$X$ is the gradient field of a smooth function $f:M\longrightarrow\RR$, then
(\ref{defrs}) becomes $\ricci(g)=cg+2\Hess(f)$ and $g$ is called a {\it gradient}
Ricci soliton. In any case, we see that Ricci solitons are very natural
generalizations of {\it Einstein} metrics (i.e. $\ricci(g)=cg$).

The main significance, though, of the concept is that $g$ is a Ricci soliton if and
only if the curve of metrics
\begin{equation}\label{rssol}
g(t)=(-2ct+1)\vp_t^*g,
\end{equation}
is a solution to the Ricci flow
\begin{equation}\label{rf}
\dpar g(t)=-2\ricci(g(t)),
\end{equation}
for some one-parameter group $\vp_t$ of diffeomorphisms of $M$.  In other words, the
Ricci flow starting at $g$ stays forever in the space of metrics which are
homothetic to $g$; it is unable to `improve' $g$.  According to (\ref{rssol}), Ricci
solitons are called {\it expanding}, {\it steady}, or {\it shrinking} depending on
whether $c<0$, $c=0$, or $c>0$.

We are interested in the following question:

\begin{quote}
Which homogeneous manifolds $G/K$ admit a $G$-invariant Ricci soliton metric?
\end{quote}

Unfortunately, even Einstein homogeneous manifolds are still not well understood
(see \cite{BhmWngZll} and \cite{cruzchica} for the compact and noncompact cases,
respectively).  But let us first review to what extent the Ricci soliton condition
is weaker than the Einstein condition for homogeneous manifolds.

Let $(M,g)$ be a Ricci soliton and let us assume it is homogeneous, i.e. its
isometry group acts transitively on $M$.  In particular, $g$ has bounded curvature.
If $g$ is steady, it is easy to see from the ODE that the scalar curvature
$\scalar(g(t))$ satisfies that $\ricci(g)=0$, and consequently $g$ must be flat (see
\cite{AlkKml}).  In the shrinking case, it follows from \cite[Theorem 1.2]{Nbr} that
$g$ is of gradient type, and it is proved in \cite{PtrWyl} that any homogeneous
gradient Ricci soliton is isometric to a quotient of $N\times\RR^k$, where $N$ is
some homogeneous Einstein manifold with positive scalar curvature and so compact and
with $\Pi_1(N)$ finite (see also \cite{Wyl}).  Finally, if $g$ is expanding then $M$
must be noncompact (see \cite{Ivy}).  Recall also that it follows from \cite{Prl}
that on a compact manifold all Ricci solitons are of gradient type.

We conclude that,

\begin{quote}
the noncompact expanding case is the only one allowing nontrivial homogeneous Ricci
solitons, and furthermore, they can not be of gradient type.
\end{quote}

All known examples so far of nontrivial homogeneous Ricci solitons are isometric to
a left-invariant metric $g$ on a simply connected Lie group $G$ (see Remark \ref{sc}
concerning non-simply connected Lie groups), and can be obtained in the following
way. Assume that $g$, which is identified with an inner product on the Lie algebra
$\ggo$ of $G$, satisfies
\begin{equation}\label{rsD}
\Ricci(g)=cI+D, \qquad\mbox{for some}\quad c\in\RR,\quad D\in\Der(\ggo),
\end{equation}
where $\Ricci(g)$ also denotes the {\it Ricci operator} of $g$ (i.e.
$\ricci(g)=g(\Ricci(g)\cdot,\cdot)$).  If $X_D\in\chi(G)$ is defined by
\begin{equation}\label{defxd}
X_D(p)=\ddt|_0\vp_t(p), \qquad p\in G,
\end{equation}
where $\vp_t\in\Aut(G)$ is the unique automorphism such that
$\dif\vp_t|_e=e^{tD}\in\Aut(\ggo)$ (the existence of $\vp_t$ follows from $G$ being
simply connected), then
$$
\Lie_{X_D}g=\ddt|_0\vp_t^*g=\ddt|_0g(e^{-tD}\cdot,e^{-tD}\cdot) = -2g(D\cdot,\cdot).
$$
This implies that the Ricci tensor equals $\ricci(g)=cg-\unm\Lie_{X_{D}}g$, and
henceforth $g$ is a Ricci soliton.  These vector fields $X_D$'s can be viewed as a
generalization to any Lie group of the so called {\it linear vector fields} on
$\RR^n$ (i.e. $X(p)=Ap$, $A\in\glg_n(\RR)$), and they play a nice and important role
in control theory (see \cite{AylTir}).  We notice that for the Gaussian soliton on
$\RR^n$ one uses the linear vector field $X(p)=cp$.

Condition (\ref{rsD}) nicely combines the geometric and algebraic features of a
left-invariant metric on a Lie group, providing a neat way to find examples of
homogeneous Ricci solitons.  These examples first appeared in \cite{soliton} ($G$
nilpotent), \cite{BrdDnl} ($G$ solvable, $\dim{G}=3$) and \cite{IsnJckLu} ($G$
solvable, $\dim{G}=4$).

\begin{remark}
It is an open question whether any Ricci soliton left invariant metric will satisfy
(\ref{rsD}), and concerning existence, we do not know of any non-solvable Lie group
admitting a nontrivial Ricci soliton.
\end{remark}

In the case when $G$ is nilpotent, metrics for which (\ref{rsD}) holds are called
{\it nilsolitons} and are known to satisfy the following properties (see the recent
survey \cite{cruzchica} for further information on nilsolitons):

\begin{itemize}
\item[(a)] Any left invariant Ricci soliton on $G$ is a nilsoliton.

\item[(b)] A given $G$ can admit at most one nilsoliton up to isometry and scaling among all its left-invariant metrics.

\item[(c)] Nilsolitons are also characterized by the following extremal property:
$$
||\ricci(g)||=\min\left\{ ||\ricci(g')||: g'\;\mbox{left-invariant on $G$},
\;\scalar(g')=\scalar(g)\right\}.
$$
Furthermore, they are the critical points of the functional square norm of the Ricci
tensor on the space of all nilmanifolds of a given dimension and scalar curvature.

\item[(d)] Nilsolitons are precisely the nilpotent parts of Einstein solvmanifolds.
\end{itemize}
Nevertheless, the existence, structural and classification problems on nilsolitons
seem to be far from being satisfactory solved, if at all possible.

\begin{definition}\label{defsols}
A left-invariant metric $g$ on a simply connected solvable Lie group is called a
{\it solsoliton} if the corresponding Ricci operator satisfies (\ref{rsD}).
\end{definition}

The name is inspired by the $3$-dimensional homogeneous geometry Sol from the
Geometrization Conjecture.  It is natural to ask, in the case of solsolitons, for
properties analogous to (a)-(d) above.  We will consider properties (b) and (d) here
and leave (a) and (c) for a forthcoming paper. Concerning property (a), it is worth
mentioning that any left invariant Ricci soliton on a {\it completely solvable} Lie
group (i.e. the eigenvalues of any $\ad{X}$ are all real) is necessarily a
solsoliton.  This follows analogously to the proof of \cite[Proposition
1.1]{soliton} by using that two left-invariant metrics on one of these groups are
isometric if and only if there is an isomorphism which is an isometry between them
(see \cite{Alk}).

\section{Variety of nilpotent Lie algebras}\label{Tb}

Let $G$ be a real reductive group acting linearly on a finite dimensional real
vector space $V$ via $(g,v)\mapsto g.v$, $g\in G,v\in V$. The precise setting is the
one in \cite{RchSld}. We also refer to \cite{EbrJbl, HnzSchStt}, where many results from
geometric invariant theory are adapted and proved over $\RR$. The Lie algebra $\ggo$
of $G$ also acts linearly on $V$ by the derivative of the above action, which will
be denoted by $(\alpha,v)\mapsto\pi(\alpha)v$, $\alpha\in\ggo$, $v\in V$. We
consider a Cartan decomposition $\ggo=\kg\oplus\pg$, where $\kg$ is the Lie algebra
of a maximal compact subgroup $K$ of $G$. Endow $V$ with a from now on fixed
$K$-invariant inner product $\ip$ such that $\pg$ acts by symmetric operators, and
endow $\pg$ with an $\Ad(K)$-invariant inner product, which will be also denoted by
$\ip$.

The function $m:V\smallsetminus\{ 0\}\longrightarrow\pg$ implicitly defined by
\begin{equation}\label{defmm}
\la m(v),\alpha\ra=\isn\la\pi(\alpha)v,v\ra, \qquad \forall\alpha\in\pg,\; v\in V,
\end{equation}

\noindent is called the {\it moment map} for the representation $V$ of $G$.  It is
easy to see that $m(cv)=m(v)$ for any nonzero $c\in\RR$ and $m$ is $K$-equivariant:
$m(k.v)=\Ad(k)m(v)$ for all $k\in K$.  In the complex case (i.e. for a complex
representation of a complex reductive algebraic group), under the natural
identifications $\pg=\pg^*=(\im\kg)^*=\kg^*$, the function $m$ is precisely the
moment map from symplectic geometry, corresponding to the Hamiltonian action of $K$
on the symplectic manifold $\PP V$ (see \cite[Chapter 8]{Mmf} for further
information).

The functional square norm of the moment map,
\begin{equation}\label{norm}
F:V\smallsetminus\{ 0\}\mapsto\RR, \qquad  F(v)=||m(v)||^2,
\end{equation}
is scaling invariant, so it can actually be viewed as a function on any sphere of
$V$ or on the projective space $\PP V$.  If $\cca$ denotes the set of critical
points of $F:V\smallsetminus\{ 0\}\longrightarrow\RR$, then it is proved in \cite{Mrn}
(see \cite{Krw1,Nss} for the complex case) that $v\in\cca$ (or equivalently, $v$ is
a minimum for $F|_{G.v}$) if and only if $v$ is an eigenvector of $\pi(m(v))$, and
in that case, the following uniqueness result holds: $G.v\cap\cca=K.v$ (up to
scaling).  This was previously proved in \cite{RchSld} (see \cite{KmpNss} for the
complex case) for the zeroes of $m$ (or equivalently, minimal vectors), which can
only appear inside closed orbits (see \cite[Section 11]{cruzchica} for a more
complete overview).

Let us consider the space of all skew-symmetric algebras of dimension $n$, which is
parameterized by the vector space
$$
V:=\lam=\{\mu:\RR^n\times\RR^n\longrightarrow\RR^n : \mu\; \mbox{bilinear and
skew-symmetric}\}.
$$
Then
$$
\nca=\{\mu\in V:\mu\;\mbox{satisfies Jacobi and is nilpotent}\}
$$
is an algebraic subset of $V$ as the Jacobi identity and the nilpotency condition
can both be written as zeroes of polynomial functions.  $\nca$ is often called the
{\it variety of nilpotent Lie algebras} (of dimension $n$).  There is a natural
linear action of $G=\G$ on $V$ given by
\begin{equation}\label{action}
g.\mu(X,Y)=g\mu(g^{-1}X,g^{-1}Y), \qquad X,Y\in\RR^n, \quad g\in\G,\quad \mu\in V.
\end{equation}

Recall that $\nca$ is $\G$-invariant and the Lie algebra isomorphism classes are
precisely the $\G$-orbits. The action of $\g$ on $V$ obtained by differentiation of
(\ref{action}) is given by
\begin{equation}\label{actiong}
\pi(\alpha)\mu=\alpha\mu(\cdot,\cdot)-\mu(\alpha\cdot,\cdot)-\mu(\cdot,\alpha\cdot),
\qquad \alpha\in\g,\quad\mu\in V.
\end{equation}

We note that $\pi(\alpha)\mu=0$ if and only if $\alpha\in\Der(\mu)$, the Lie algebra
of derivations of the algebra $\mu$.  The canonical inner product $\ip$ on $\RR^n$
determines an $\Or(n)$-invariant inner product on $V$, also denoted by $\ip$, as
follows:
\begin{equation}\label{innV}
\la\mu,\lambda\ra= \sum\la\mu(e_i,e_j),\lambda(e_i,e_j)\ra
=\sum\la\mu(e_i,e_j),e_k\ra\la\lambda(e_i,e_j),e_k\ra,
\end{equation}

\noindent and also the standard $\Ad(\Or(n))$-invariant inner product on $\g$ given
by
\begin{equation}\label{inng}
\la \alpha,\beta\ra=\tr{\alpha \beta^{\mathrm t}}=\sum\la\alpha e_i,\beta e_i\ra
=\sum\la\alpha e_i,e_j\ra\la\beta e_i,e_j\ra, \qquad \alpha,\beta\in\g,
\end{equation}

\noindent where $\{ e_1,...,e_n\}$ denotes the canonical basis of $\RR^n$.  We have
made several abuses of notation concerning inner products.  Recall that $\ip$ has
been used to denote inner products on $\RR^n$, $V$ and $\g$ indistinctly.  We note
that $\pi(\alpha)^t=\pi(\alpha^t)$ and $(\ad{\alpha})^t=\ad{\alpha^t}$ for any
$\alpha\in\g$, due to the choice of these canonical inner products everywhere.

We use $\g=\sog(n)\oplus\sym(n)$ as a Cartan decomposition for the Lie algebra $\g$
of $\G$, where $\sog(n)$ and $\sym(n)$ denote the subspaces of skew-symmetric and
symmetric matrices, respectively.  It is proved in \cite[Proposition 3.5]{minimal}
that if $\ad_{\mu}$ denotes left multiplication for the algebra $\mu$ then the
moment map $m:V\smallsetminus\{ 0\}\longrightarrow\sym(n)$ for the action (\ref{action})
is given by
$$
m(\mu)=\tfrac{1}{||\mu||^2} \left(-2\sum(\ad_{\mu}{e_i})^t\ad_{\mu}{e_i} +
\sum\ad_{\mu}{e_i}(\ad_{\mu}{e_i})^t\right),
$$
or equivalently, for all $X\in\RR^n$,
\begin{equation}\label{mmv}
\la m(\mu)X,X\ra=\tfrac{1}{||\mu||^2}\left(-2\sum\la\mu(X,e_i),e_j\ra^2
+\sum\la\mu(e_i,e_j),X\ra^2\right).
\end{equation}

Let $\tg$ denote the set of all diagonal $n\times n$ matrices.  If $\{
e_1',...,e_n'\}$ is the basis of $(\RR^n)^*$ dual to the canonical basis $\{
e_1,...,e_n\}$ then
$$
\{ v_{ijk}=(e_i'\wedge e_j')\otimes e_k : 1\leq i<j\leq n, \; 1\leq k\leq n\}
$$
is a basis of weight vectors of $V$ for the action (\ref{action}), where $v_{ijk}$
is actually the bilinear form on $\RR^n$ defined by
$v_{ijk}(e_i,e_j)=-v_{ijk}(e_j,e_i)=e_k$ and zero otherwise.  The corresponding
weights $\alpha_{ij}^k\in\tg$, $i<j$, are given by
\begin{equation}\label{alfas}
\pi(\alpha)v_{ijk}=(a_k-a_i-a_j)v_{ijk}=\la\alpha,\alpha_{ij}^k\ra v_{ijk},
\quad\forall\alpha=\left[\begin{smallmatrix} a_1&&\\ &\ddots&\\ &&a_n
\end{smallmatrix}\right]\in\tg,
\end{equation}

\noindent where $\alpha_{ij}^k=E_{kk}-E_{ii}-E_{jj}$ and $\ip$ is the inner product
defined in (\ref{inng}).  As usual $E_{rs}$ denotes the matrix with $1$ at entry
$rs$ and $0$ elsewhere.

From now on, we will always denote by $\mu_{ij}^k$ the structure constants of a
vector $\mu\in V$ with respect to the basis $\{ v_{ijk}\}$:
$$
\mu=\sum\mu_{ij}^kv_{ijk}, \qquad \mu_{ij}^k\in\RR, \qquad {\rm i.e.}\quad
\mu(e_i,e_j)=\sum_{k=1}^n\mu_{ij}^ke_k, \quad i<j.
$$
Each nonzero $\mu\in V$ uniquely determines an element $\beta_{\mu}\in\tg$ given by
$$
\beta_{\mu}:=\mcc\left\{\alpha_{ij}^k:\mu_{ij}^k\ne 0\right\},
$$
where $\mcc(X)$ denotes the unique element of minimal norm in the convex hull
$\CH(X)$ of a subset $X\subset\tg$.  We note that $\beta_{\mu}$ is always nonzero
since $\tr{\alpha_{ij}^k}=-1$ for all $i<j$ and consequently $\tr{\beta_{\mu}}=-1$.

Let $\tg^+$ denote the Weyl chamber of $\g$ given by
\begin{equation}\label{weyl}
\tg^+=\left\{\left[\begin{smallmatrix} a_1&&\\ &\ddots&\\ &&a_n
\end{smallmatrix}\right]\in\tg:a_1\leq...\leq a_n\right\}.
\end{equation}

In \cite{standard}, a $\G$-invariant stratification for $V=\lam$ has been defined by
adapting to this context the construction given in \cite[Section 12]{Krw1} for
reductive group representations over an algebraically closed field.  We summarize in
the following theorem the main properties of the stratification, which will become
our main tool in the study of the structure of solsolitons in next section (see
\cite[Section 7]{cruzchica} for a detailed overview).

\begin{theorem}\label{strata}\cite{standard, einsteinsolv}
There exists a finite subset $\bca\subset\tg^+$, and for each $\beta\in\bca$ a
$\G$-invariant subset $\sca_{\beta}\subset V$ (a {\it stratum}) such that
$$
V\smallsetminus\{ 0\}=\bigcup_{\beta\in\bca}\sca_{\beta} \qquad \mbox{(disjoint
union)},
$$
and $\tr{\beta}=-1$ for any $\beta\in\bca$.  For $\mu\in\sca_{\beta}$ we have that
\begin{equation}\label{betapos}
\beta+||\beta||^2I \quad\mbox{is positive definite for all}\; \beta\in\bca
\;\mbox{such that}\; \sca_{\beta}\cap\nca\ne\emptyset,\;\mbox{and}
\end{equation}
\begin{equation}\label{bmu}
||\beta||\leq ||m(\mu)||\qquad(\mbox{here, equality holds}\;\Leftrightarrow
m(\mu)\;\mbox{is conjugate to}\; \beta).
\end{equation}

If in addition, $\mu\in\sca_{\beta}$ satisfies $\beta_{\mu}=\beta$ (or equivalently,
$\min\left\{\la\beta,\alpha_{ij}^k\ra:\mu_{ij}^k\ne 0\right\}=||\beta||^2$), which
always holds for some $g.\mu$, $g\in\Or(n)$, then 
\begin{equation}\label{adbeta}
\left\la[\beta,D],D\right\ra\geq 0 \qquad\forall\; D\in\Der(\mu) \qquad(\mbox{here,
equality holds}\;\Leftrightarrow [\beta,D]=0),
\end{equation}
\begin{equation}\label{betaort}
\tr{\beta D}=0 \quad\forall\; D\in\Der(\mu), \;\mbox{and}
\end{equation}
\begin{equation}\label{delta}
\left\la\pi\left(\beta+||\beta||^2I\right)\mu,\mu\right\ra\geq 0 \qquad (\mbox{here,
equality holds}\;\Leftrightarrow \beta+||\beta||^2I\in\Der(\mu)).
\end{equation}
\end{theorem}

This stratification is based on instability results and is strongly related to the
moment map in many ways other than (\ref{bmu})  (see \cite{cruzchica}).

\section{Structure of solsolitons}\label{structure}

Let $S$ be a {\it solvmanifold}, that is, a simply connected solvable Lie group
endowed with a left invariant Riemannian metric.  $S$ will be often identified with
its metric Lie algebra $(\sg,\ip)$, where $\sg$ is the Lie algebra of $S$ and $\ip$
denotes the inner product on $\sg$ which determines the metric.  We consider the
orthogonal decomposition
$$
\sg=\ag\oplus\ngo,
$$
where $\ngo$ is the nilradical of $\sg$ (i.e. maximal nilpotent ideal).  The {\it
mean curvature vector} of $S$ is the only element $H\in\ag$ such that $\la
H,A\ra=\tr{\ad{A}}$ for any $A\in\ag$.  If $B$ denotes the symmetric map defined by
the Killing form of $\sg$ relative to $\ip$ (i.e. $\la BX,X\ra=\tr{(\ad{X})^2}$ for
all $X\in\sg$) then $B(\ag)\subset\ag$ and $B|_{\ngo}=0$.  The Ricci operator
$\Ricci$ of $S$ is given by (see for instance \cite[7.38]{Bss}):
\begin{equation}\label{ricci}
\Ricci=R-\unm B-S(\ad{H}),
\end{equation}

\noindent where $S(\ad{H})=\unm(\ad{H}+(\ad{H})^t)$ is the symmetric part of
$\ad{H}$ and $R$ is the symmetric operator defined by
\begin{equation}\label{R}
\la RX,X\ra=-\unm\sum\la [X,X_i],X_j\ra^2 +\unc\sum\la
[X_i,X_j],X\ra^2,\qquad\forall X\in\sg,
\end{equation}

\noindent where $\{ X_i\}$ is any orthonormal basis of $(\sg,\ip)$.

It follows from (\ref{mmv}) that this anonymous tensor $R$ in the formula of the
Ricci operator satisfies
\begin{equation}\label{mmR}
m([\cdot,\cdot])=\tfrac{4}{||[\cdot,\cdot]||^2}R,
\end{equation}

\noindent where $m:\Lambda^2\sg^*\otimes\sg\longrightarrow\sym(\sg)$ is the moment
map for the natural action of $\Gl(\sg)$ on $\Lambda^2\sg^*\otimes\sg$.  In other
words, $R$ may be alternatively defined as follows:
\begin{equation}\label{Rmm}
\tr{RE}=\unc\la \pi(E)[\cdot,\cdot],[\cdot,\cdot]\ra, \qquad\forall E\in\End(\sg),
\end{equation}

\noindent where we are considering the Lie bracket $[\cdot,\cdot]$ of $\sg$ as a
vector in $\Lambda^2\sg^*\otimes\sg$, $\ip$ is the inner product defined in
(\ref{innV}) and $\pi$ is the representation given in (\ref{actiong}) (see the
notation in Section \ref{Tb} and replace $\RR^n$ by $\sg$).

\begin{remark}\label{st2}
In particular, $R$ is orthogonal to any derivation of $\sg$.  It is easy to see that
the same holds for $B$.
\end{remark}

\begin{remark}\label{st1}
If $\sg$ is nilpotent then $\Ricci= R$ and hence the scalar curvature is simply
given by $\scalar=\tr{R}=-\unc ||\lb||^2$.
\end{remark}

The following more explicit formula for the Ricci operator of $S$ follows from a
straightforward computation by using (\ref{ricci}) and (\ref{R}):
\begin{equation}\label{ricgen}
\begin{array}{rcl}
\la\Ricci A,A\ra &=& -\unm\sum ||[A,A_i]||^2 -\tr{S(\ad{A}|_{\ngo})^2}, \\ \\

\la\Ricci A,X\ra &=& -\unm\sum\la[A,A_i],[X,A_i]\ra-\unm\tr{(\ad{A}|_{\ngo})^t\ad{X}|_{\ngo}} \\ \\ && -\unm\la[H,A],X\ra, \\ \\
\la\Ricci X,X\ra &=& \unc\sum\la[A_i,A_j],X\ra^2
+\unm\sum\la[\ad{A_i}|_{\ngo},(\ad{A_i}|_{\ngo})^t](X),X\ra \\ \\ && -\unm\sum\la
[X,X_i],X_j\ra^2 +\unc\sum\la [X_i,X_j],X\ra^2 -\la[H,X],X\ra,
\end{array}
\end{equation}

\noindent for all $A\in\ag$ and $X\in\ngo$, where $\{ A_i\}$, $\{ X_i\}$, are any
orthonormal basis of $\ag$ and $\ngo$, respectively, and
$S(\ad{A}|_{\ngo})=\unm(\ad{A}|_{\ngo}+(\ad{A}|_{\ngo})^t)$. It is now clear from
(\ref{ricgen}) that there is a substantial simplification of the expression of
$\Ricci$ under the assumptions $[\ag,\ag]=0$ and $\ad{A}$ symmetric for all
$A\in\ag$.  This gives rise to the following natural construction of solsolitons
starting from a nilsoliton.

\begin{proposition}\label{const}
Let $(\ngo,\ip_1)$ be a nilsoliton, say with Ricci operator $\Ricci_1=cI+D_1$,
$c<0$, $D_1\in\Der(\ngo)$, and consider $\ag$ any abelian Lie algebra of symmetric
derivations of $(\ngo,\ip_1)$.  Then the solvmanifold $S$ with Lie algebra
$\sg=\ag\oplus\ngo$ (semidirect product) and inner product given by
$$
\ip|_{\ngo\times\ngo}=\ip_1, \qquad \la\ag,\ngo\ra=0, \qquad \la
A,A\ra=-\tfrac{1}{c}\tr{A^2} \qquad\forall A\in\ag,
$$
is a solsoliton with $\Ricci=cI+D$, where $D\in\Der(\sg)$ is defined by
$D|_{\ag}=0$, $D|_{\ngo}=D_1-\ad{H}|_{\ngo}$ and $H$ is the mean curvature vector of
$S$. Furthermore, $S$ is Einstein if and only if $D_1\in\ag$.
\end{proposition}

\begin{remark} The aim of this section is to show that this very simple procedure
actually yields all solsolitons up to isometry.
\end{remark}

\begin{remark} If $\ngo$ is abelian then $\Ricci_1=0$ and so we can take $D_1=-cI$ for any $c<0$.
\end{remark}

\begin{proof}
It follows directly from the hypotheses and (\ref{ricgen}) that $\Ricci|_{\ag}=cI$
and
$$
\Ricci|_{\ngo}=\Ricci_1-H=cI+D_1-\ad{H}|_{\ngo}.
$$
If $\ngo=\ngo_1\oplus...\oplus\ngo_r$ is the orthogonal decomposition with
$[\ngo,\ngo]=\ngo_2\oplus...\oplus\ngo_r$,
$[\ngo,[\ngo,\ngo]]=\ngo_3\oplus...\oplus\ngo_r$, and so on, then for any $A\in\ag$,
$\ad{A}|_{\ngo}=A$ leaves the subspaces $\ngo_i$ invariant as it is a symmetric
derivation (it is actually enough for this that $(\ad{A}|_{\ngo})^t$ be a derivation
as well) and $\ad{X}(\ngo_i)\subset\ngo_{i+1}\oplus...\oplus\ngo_r$ for all $i$.
This implies that $\tr{(\ad{A}|_{\ngo})^t\ad{X}|_{\ngo}}=0$, and thus
$\la\Ricci\ag,\ngo\ra=0$.  It only remains to prove that $D\in\Der(\sg)$, for which
it is enough to show that $[\ag,D_1]=0$, but this follows by using that any
symmetric derivation of $(\ngo,\ip_1)$ (and more generally any derivation whose
transpose is also a derivation) commutes with $\Ricci_1$ (see \cite[Lemma
2.2]{Hbr}).

Let us now prove the Einstein assertion.  If $S$ is Einstein then
$D_1=\ad{H}|_{\ngo}=H\in\ag$.  Conversely, if $D_1\in\ag$, then since
$\tr{\Ricci_1A}=0$ for any $A\in\ag$ (see Remarks \ref{st1} and \ref{st2}) we have
that
$$
\tr{D_1A}=-c\tr{A}=-c\tr{\ad{A}}=-c\la H,A\ra =\tr{HA}=\tr{\ad{H}|_{\ngo}A},
\qquad\forall A\in\ag.
$$
This implies that $D_1=\ad{H}|_{\ngo}$ (i.e. $\Ricci=cI$), completing the proof of
the proposition.
\end{proof}

We now show that $c\geq 0$ is actually not allowed for nontrivial solsolitons.  This
gives an alternative proof of the fact that any nonflat solsoliton must be expanding
(see Section \ref{hrs}).

\begin{proposition}\label{cneg}
Let $S$ be a solsoliton, say with $\Ricci=cI+D$, $c\in\RR$, $D\in\Der(\sg)$. If
$c\geq 0$ then $\Ricci=0$.
\end{proposition}

\begin{proof}
Since $D$ is a symmetric derivation we have that $D|_{\ag}=0$, $D(\ngo)\subset\ngo$.
It then follows from just the first line in (\ref{ricgen}) that $c=0$, $[\ag,\ag]=0$
and $\ad{A}^t=-\ad{A}$ for any $A\in\ag$.  We also obtain from (\ref{ricgen}) that
$D|_{\ngo}=\Ricci_1$, the Ricci operator of $(\ngo,\ip_{\ngo\times\ngo})$, and so
$\Ricci_1=0$ by Remarks \ref{st1} and \ref{st2}, concluding the proof.
\end{proof}

We will need in what follows the following technical lemma valid for general
solvmanifolds.

\begin{lemma}\label{adA}
Let $S$ be a solvmanifold.  Then, for any $A\in\ag$, the following conditions are
equivalent:
\begin{itemize}
\item[(i)] $(\ad{A})^t$ is a derivation of $\sg$.

\item[(ii)] $\ad{A}$ is a normal operator (i.e. $[\ad{A},(\ad{A})^t]=0$).
\end{itemize}
\end{lemma}

\begin{proof}
It follows from (\ref{Rmm}) that
$$
\begin{array}{rcl}
\tr{R[\ad{A},(\ad{A})^t]} &=& \unc\la\pi([\ad{A},(\ad{A})^t])[\cdot,\cdot],[\cdot,\cdot]\ra \\ \\
&=& \unc\la\pi(\ad{A})\pi((\ad{A})^t)[\cdot,\cdot],[\cdot,\cdot]\ra \\ \\
&=& \unc\la\pi((\ad{A})^t)[\cdot,\cdot],\pi(\ad{A})^t[\cdot,\cdot]\ra \\ \\
&=& \unc||\pi((\ad{A})^t)[\cdot,\cdot]||^2,
\end{array}
$$
and so if $\ad{A}$ is normal then $\pi((\ad{A})^t)\lb=0$, that is,
$(\ad{A})^t\in\Der(\sg)$.

Conversely, if $(\ad{A})^t$ is a derivation of $\sg$ then both $\ad{A}$ and
$(\ad{A})^t$ must vanish on $\ag$, since they leave $\ngo$ invariant and their
images are contained in $\ngo$ (this last statement is well-known to be true for any
derivation, see for instance \cite[Lemma 2.6]{GrdWls}).  This implies that
$$
(\ad{A})^t([A,X])=[(\ad{A})^t(A),X]+[A,(\ad{A})^t(X)]= [A,(\ad{A})^t(X)],
\qquad\forall X\in\ngo,
$$
which is equivalent to saying that $[\ad{A},(\ad{A})^t]=0$.
\end{proof}

The following structural theorem for solsolitons is the main result of this paper.

\begin{theorem}\label{main}
Let $S$ be a solvmanifold with metric Lie algebra $(\sg,\ip)$ and consider the
orthogonal decomposition $\sg=\ag\oplus\ngo$, where $\ngo$ is the nilradical of
$\sg$.  Then $\Ricci=cI+D$ for some $c<0$, $D\in\Der(\sg)$, i.e. $S$ is a
solsoliton, if and only if the following conditions hold:

\begin{itemize}
\item[(i)] $(\ngo,\ip|_{\ngo\times\ngo})$ is a nilsoliton with Ricci
operator $\Ricci_1=cI+D_1$, for some $D_1\in\Der(\ngo)$.

\item[(ii)] $[\ag,\ag]=0$.

\item[(iii)] $(\ad{A})^t\in\Der(\sg)$ (or equivalently,
$[\ad{A},(\ad{A})^t]=0$) for all $A\in\ag$.

\item[(iv)] $\la A,A\ra=-\tfrac{1}{c}\tr{S(\ad{A})^2}$ for all $A\in\ag$,
where $S(\ad{A})=\unm(\ad{A}+(\ad{A})^t)$.
\end{itemize}
\end{theorem}

\begin{proof}
If conditions (i)-(iv) are satisfied by $S$ then we can argue exactly as in the
proof of Proposition \ref{const} to obtain that $\Ricci=cI+D$, where $D$ is defined
by $D|_{\ag}=0$, $D|_{\ngo}=D_1-S(\ad{H}|_{\ngo})$, $H$ the mean curvature vector of
$S$.

Let us then prove the converse assertion.  Let $S$ be a solvmanifold with
$\Ricci=cI+D$, $c<0$, $D\in\Der(\sg)$.  If $F:=S(\ad{H})+D$ then we obtain from
(\ref{ricci}) and (\ref{Rmm}) that
\begin{equation}\label{einstein}
\tr{\left(cI+\unm B+F\right)E}= \unc\la\pi(E)\lb,\lb\ra, \qquad\forall
E\in\End(\sg).
\end{equation}

By letting $E=\ad{H}+D$ in (\ref{einstein}) and using Remark \ref{st2} we get
\begin{equation}\label{c}
c\tr{F}+\tr{F^2}=0.
\end{equation}
In particular, $\tr{F}\geq 0$ and equality holds if and only if $F=0$.  Also, by
applying (\ref{einstein}) to $E$ defined by $E|_{\ag}=0$ and $E|_{\ngo}=I$, it is
easy to see that
\begin{equation}\label{c2}
cn+\tr{F}=\unc\sum||[A_i,A_j]||^2+\unc||\lb_{\ngo\times\ngo}||^2, \qquad
n=\dim{\ngo}.
\end{equation}

We first consider the case when $\ngo$ is abelian.  It follows from (\ref{c2}) that
$cn+\tr{F}\geq 0$, and so by (\ref{c}) we get that
$\tr{F^2}\leq\tfrac{1}{n}(\tr{F})^2$. This implies that $[\ag,\ag]=0$, since
equality must hold in (\ref{c2}), and also that $F|_{\ngo}=tI$ for some $t\geq 0$,
but therefore $F|_{\ag}=0$ and $F|_{\ngo}=-cI$. We now obtain from (\ref{ricgen})
that the restrictions to $\ngo$ satisfy
$$
cI+D = \unm\sum[\ad{A_i},(\ad{A_i})^t]-S(\ad{H}),
$$
from which it follows that $\sum[\ad{A_i},(\ad{A_i})^t]=0$.  By arguing as in the
proof of Lemma \ref{adA} we get that
\begin{equation}\label{trRadA}
0=\tr\left(R\sum[\ad{A_i},(\ad{A_i})^t]\right)=\unc\sum||\pi((\ad{A_i})^t)\lb||^2,
\end{equation}

\noindent and therefore (iii) follows.

Let us assume from now on that $\ngo$ is not abelian.  In order to apply the results
in Section \ref{Tb}, we identify $\ngo$ with $\RR^n$ via an orthonormal basis $\{
e_1,...,e_n\}$ of $\ngo$.  In this way, $\mu:=\lb|_{\ngo\times\ngo}$ can be viewed
as a nonzero element of $\nca\subset V$.  Thus $\mu\in\sca_{\beta}$ for some
$\beta\in\bca\subset\tg^+$  and there exists $g\in\Or(n)$ such that
$\tilde{\mu}:=g.\mu$ satisfies $\beta_{\tilde{\mu}}=\beta$, so that in addition
(\ref{betaort}) and (\ref{delta}) hold for $\tilde{\mu}$ (see Theorem \ref{strata}).
Let $\tilde{g}$ denote the orthogonal map of $(\sg,\ip)$ defined by
$\tilde{g}|_{\ag}=I$, $\tilde{g}|_{\ngo}=g$, and let $\tilde{\sg}$ be the Lie
algebra which is $\sg$ as a vector space and has Lie bracket
$$
\tilde{g}.[\cdot,\cdot]=\tilde{g}[\tilde{g}^{-1}\cdot,\tilde{g}^{-1}\cdot].
$$
We therefore have that $(\tilde{\sg},\ip)$ is isometric to $(\sg,\ip)$, as
$\tilde{g}:\ggo\longrightarrow\tilde{\ggo}$ is an isometric isomorphism between the
two metric Lie algebras. Since conditions (i)-(iv) hold for $(\tilde{\sg},\ip)$ if
and only if they hold for $(\sg,\ip)$, we can assume from now on that all properties
(\ref{adbeta})-(\ref{delta}) in Theorem \ref{strata} hold for $\mu$.

Consider $E_{\beta}\in\End(\sg)$ defined by
$$
E_{\beta}:=\left[\begin{array}{cc} 0&0\\
0&\beta+||\beta||^2I\end{array}\right], \qquad \mbox{that is}, \quad
E_{\beta}|_{\ag}=0, \quad E_{\beta}|_{\ngo}=\beta+||\beta||^2I.
$$

\begin{lemma}\label{est}
$\la\pi(E_{\beta})\lb,\lb\ra\geq 0$ and it equals the sum of the following three
nonnegative terms:
\begin{equation}\label{ests}
\begin{array}{l}
 \unc\sum\la(\beta+||\beta||^2I)[A_i,A_j],[A_i,A_j]\ra \\ \\
 + \unm\sum\la [\beta,\ad{A_i}],\ad{A_i}\ra \\ \\
 + \unc\la\pi(\beta+||\beta||^2I)\mu,\mu\ra.
\end{array}
\end{equation}
\end{lemma}

\begin{proof}
If $\{ A_r\}$ is an orthonormal basis of $\ag$ then
$$
\begin{array}{rcl}
\la\pi(E_{\beta})\lb,\lb\ra &=& \unc\sum\la E_{\beta}[A_r,A_s],[A_r,A_s]\ra \\ \\
&&+ \unm\sum\la E_{\beta}[A_r,e_i],[A_r,e_i]\ra - \unm\sum\la
[A_r,E_{\beta}e_i],[A_r,e_i]\ra, \\ \\
&& + \unc\sum\la E_{\beta}[e_i,e_j]-[E_{\beta}e_i,e_j]
-[e_i,E_{\beta}e_j],[e_i,e_j]\ra,
\end{array}
$$
which in turn equals
\begin{equation}\label{prueba2}
\begin{array}{l}
\unc\sum\la(\beta+||\beta||^2I)[A_r,A_s],[A_r,A_s]\ra \\ \\
 + \unm\sum\la (\beta\ad{A_r}-\ad{A_r}\beta)(e_i),\ad{A_r}(e_i)\ra
 +\unc\la\pi(\beta+||\beta||^2I)\mu,\mu\ra,
\end{array}
\end{equation}

\noindent and so (\ref{ests}) follows.  The three terms in (\ref{ests}) are $\geq 0$
by (\ref{betapos}), (\ref{adbeta}) and (\ref{delta}), respectively.
\end{proof}

We therefore obtain, by applying (\ref{einstein}) to $E=E_{\beta}$ and using Lemma
\ref{est}, that
\begin{equation}\label{prueba1}
c\tr{E_{\beta}}+\tr{FE_{\beta}}\geq 0.
\end{equation}

Recall that $\tr{\beta}=-1$ (see Theorem \ref{strata}) and so
\begin{equation}\label{prueba3}
\begin{array}{ll}
\tr{E_{\beta}^2} & =\tr(\beta^2+||\beta||^4I+2||\beta||^2\beta)=||\beta||^2(1+n||\beta||^2-2) \\ \\
& = ||\beta||^2(-1+n||\beta||^2)=||\beta||^2\tr{E_{\beta}}.
\end{array}
\end{equation}

On the other hand, we have that
\begin{equation}\label{prueba4}
\tr{FE_{\beta}}=\tr{F|_{\ngo}(\beta+||\beta||^2)}=||\beta||^2\tr{F}
\end{equation}

\noindent by (\ref{betaort}).  We now use (\ref{c}), (\ref{prueba1}),
(\ref{prueba3}) and (\ref{prueba4}) and obtain by a straightforward manipulation
that
$$
\tr{F^2}\tr{E_{\beta}^2}\leq (\tr{FE_{\beta}})^2,
$$
that is, we get a `backwards' Cauchy-Schwartz inequality.  This has many strong
consequences. A first one is that $F=tE_{\beta}$ for some $t\geq 0$, and thus by
(\ref{c}),
\begin{equation}\label{tfor}
F|_{\ngo}=-\tfrac{c}{||\beta||^2}E_{\beta}.
\end{equation}

Recall that $\tr{F}>0$ since $F\ne 0$ by (\ref{prueba1}) and (\ref{betapos}).
Secondly, the three nonnegative terms in (\ref{ests}) must vanish, which
respectively implies that $[\ag,\ag]=0$ by (\ref{betapos}),
$[\beta,\ad{\ag}|_{\ngo}]=0$ by (\ref{adbeta}) and $\beta+||\beta||^2I\in\Der(\ngo)$
by (\ref{delta}).  Thus (ii) holds and so (iv) follows from (\ref{ricgen}).

We also obtain from (\ref{ricgen}) that the restrictions to $\ngo$ satisfy
\begin{equation}\label{cID}
cI+D = \unm\sum[\ad{A_i},(\ad{A_i})^t] + \Ricci_1 - S(\ad{H}),
\end{equation}

\noindent and hence it follows from (\ref{tfor}) that
$$
\unm\sum[\ad{A_i},(\ad{A_i})^t] +\Ricci_1=-\tfrac{c}{||\beta||^2}\beta.
$$
By taking traces we get $-\unc||\mu||^2=\tfrac{c}{||\beta||^2}$ (see Remark
\ref{st1}). This implies that
$$
\begin{array}{l}
\tfrac{1}{8}\sum||\pi((\ad{A_i})^t)\lb||^2 +\tr{\Ricci_1^2} =
\tr{\Ricci_1\sum[\ad{A_i},(\ad{A_i})^t]} +\tr{\Ricci_1^2} \\ \\
= -\tfrac{c}{||\beta||^2}\tr{\Ricci_1\beta} = \unc||\mu||^2\tr{\Ricci_1\beta} =
\tfrac{1}{16}||\mu||^4\la m(\mu),\beta\ra \leq \tfrac{1}{16}||\mu||^4 ||m(\mu)||
||\beta|| \\ \\
\leq \tfrac{1}{16}||\mu||^4 ||m(\mu)||^2 = \tr{\Ricci_1^2}.
\end{array}
$$
The first equality above holds by (\ref{trRadA}) and the last inequality follows
from (\ref{bmu}) and the fact that $m(\mu)=\tfrac{4}{||\mu||^2}\Ricci_1$ (see
(\ref{mmR}) and Remark \ref{st1}).  We therefore obtain that
$$
\sum||\pi((\ad{A_i})^t)\lb||^2=0,
$$
and so (iii) follows.  We now use (\ref{cID}) and Lemma \ref{adA} to conclude that
$$
\Ricci_1=cI+D+S(\ad{H})\in\RR I +\Der(\ngo),
$$
and therefore part (i) holds.
\end{proof}

It is well-known that a solvmanifold which satisfies conditions (ii) and (iii) in
Theorem \ref{main} is isometric to the one obtained by just changing the Lie bracket
into
\begin{equation}\label{sympart}
[A,X]'=S(\ad{A})X, \qquad [X,Y]'=[X,Y], \qquad\forall A\in\ag,\; X,Y\in\ngo,
\end{equation}
and keeping the same $\ip$ (see for instance \cite[Proposition 2.5]{Hbr}).

Thus the next result follows directly from Theorem \ref{main}.

\begin{corollary}\label{cormain}
Up to isometry, any solsoliton can be constructed as in {\rm Proposition
\ref{const}}.
\end{corollary}

As a byproduct of the proof of Theorem \ref{main}, the following extra structural
properties for solsolitons have been obtained.

\begin{proposition}\label{extras}
Let $S$ be a solsoliton, say with $\Ricci=cI+D$, $c<0$, $D\in\Der(\sg)$, and let us
assume that $\ngo$ is not abelian, $\mu:=\lb|_{\ngo\times\ngo}\in\sca_{\beta}$ and
$\beta_{\mu}=\beta$.  If $E_{\beta}\in\End(\sg)$ is defined by
$$
E_{\beta}:=\left[\begin{smallmatrix} 0&0\\
0&\beta+||\beta||^2I\end{smallmatrix}\right], \qquad\mbox{i.e.}\quad
E_{\beta}|_{\ag}=0, \quad E_{\beta}|_{\ngo}=\beta+||\beta||^2I,
$$
then the following conditions hold:

\begin{itemize}
\item[(i)] $E_{\beta}\in\Der(\sg)$ (or equivalently, $[\beta,\ad{\ag}]=0$
and $\beta+||\beta||^2I\in\Der(\ngo)$).

\item[(ii)] $S(\ad{H})+D=-\tfrac{c}{||\mu||^2}E_{\beta}$.  In particular,
$S$ is Einstein if and only if $S(\ad{H})=-\tfrac{c}{||\mu||^2}E_{\beta}$.

\item[(iii)] $c=-\unc ||\mu||^2||\beta||^2$ and $m(\mu)=\beta$.
\end{itemize}
\end{proposition}

Recall that condition $\beta_\mu=\beta$ can be assumed to hold up to isometry since
always $\beta_{g.\mu}=\beta$ for some $g\in\Or(\ngo)$ (see Theorem \ref{strata} and
the paragraph before Lemma \ref{est}).

\begin{remark}\label{sc}
Can we use condition (\ref{rsD}) to find examples of Ricci solitons on solvable Lie
groups which are not simply connected?  The answer is no, such as for Einstein
solvmanifolds.  Indeed, for the field $X_D$ (see (\ref{defxd})) to be defined on a
Lie group $S/\Gamma$ covered by a simply connected solvable Lie group $S$, where
$\Gamma$ is a discrete subgroup of the center of $S$, it is necessary that the
one-parameter group of $\vp_t\in\Aut(S)$ with $\dif\vp_t|_e=e^{tD}$ satisfies
$\vp_t(\Gamma)=\Gamma$ for all $t\in\RR$.  But since $\Gamma$ is discrete this
implies that $\vp_t(\gamma)=\gamma$ for all $\gamma\in\Gamma$ and $t$, and
consequently $DX=0$ for some nonzero $X$ in the center $\zg(\sg)$ of $\sg$.  It
follows from Theorem \ref{main} (ii) and (iii) that $X\in\ngo$, and since
$D|_{\ngo}=D_1-S(\ad{H})|_{\ngo}$ we obtain that
$$
0=\la DX,X\ra=\la D_1X,X\ra-\la[H,X],X\ra = \la D_1X,X\ra,
$$
which contradicts the fact that $D_1$ is positive definite.
\end{remark}

\section{Uniqueness of solsolitons}\label{uniq}

The structural results obtained in Theorem \ref{main} pave the way for the following
uniqueness result for solsolitons, which is the analogous of the one already known
for nilsolitons.

\begin{theorem}\label{uni}
Let $S$ and $S'$ be two solsolitons which are isomorphic as Lie groups.  Then $S$ is
isometric to $S'$ up to scaling.
\end{theorem}

\begin{remark}
In particular, a given solvable Lie group can admit at most one Ricci soliton left
invariant metric up to isometry and scaling.
\end{remark}

\begin{proof}
Without any loss of generality, we can assume that $\sg=\sg'$, $\ag=\ag'$ and
$\ngo=\ngo'$ as vector spaces and that $\ip=\ip'$.  Thus the solvmanifolds $S$ and
$S'$ are respectively determined only by their Lie brackets $\lb$ and $\lb'$, which
must satisfy all structural properties in Theorem \ref{main}. From now on, such properties will
be referred to just by (i),...,(iv).  Since $S$ and $S'$ are isomorphic,
there exists a Lie algebra isomorphism between $\lb$ and $\lb'$ of the form
$$
f=\left[\begin{array}{cc} g&0\\
j&h\end{array}\right], \qquad g\in\Gl(\ag), \quad h\in\Gl(\ngo), \quad
j:\ag\longrightarrow\ngo,
$$
such that
\begin{equation}\label{ison}
h.\mu=\mu', \qquad \mu:=\lb|_{\ngo\times\ngo}, \quad \mu':=\lb'|_{\ngo\times\ngo},
\end{equation}

\noindent and
\begin{equation}\label{isoa}
h\ad(g^{-1}A)|_{\ngo}h^{-1} = \ad'{A}|_{\ngo} + \ad_{\mu'}(jg^{-1}A), \qquad\forall
A\in\ag.
\end{equation}

These two conditions are also sufficient for $f$ being an isomorphism by (ii).  We
can assume up to scaling that the scalars $c$ and $c'$ in the definition of a
solsoliton satisfy $c=c'$, and therefore $||\mu||=||\mu'||$ by Proposition
\ref{extras}, (iii) and the fact that $\mu$ and $\mu'$ belong to the same
$\Gl(\ngo)$-orbit and consequently they must lie in the same stratum $\sca_{\beta}$
(see Theorem \ref{strata}). It then follows from (i) and the uniqueness for
nilsolitons (see either \cite[Theorem 3.5]{soliton} or \cite[Theorem
4.2]{cruzchica}) that $h\in\Or(\ngo)$.

If $\ngo=\ngo_1\oplus...\oplus\ngo_r$ is the orthogonal decomposition with
$[\ngo,\ngo]=\ngo_2\oplus...\oplus\ngo_r$,
$[\ngo,[\ngo,\ngo]]=\ngo_3\oplus...\oplus\ngo_r$, and so on, then by (iii) we have
that for all $i$, $\ngo_i$ is an invariant subspace for any $\ad{A}|_{\ngo},
\ad'{A}|_{\ngo}$, $A\in\ag$, and also for $h$.  But since
$\ad_{\mu'}(jg^{-1}A)\ngo_i\subset\ngo_{i+1}\oplus...\oplus\ngo_r$ for all $i$,
condition (\ref{isoa}) implies that $\ad_{\mu'}(jg^{-1}A)=0$ for all $A\in\ag$ and
consequently
$$
\tilde{f}=\left[\begin{array}{cc} g&0\\
0&h\end{array}\right]
$$
is also an isomorphism between $\lb$ and $\lb'$.  We finally use (\ref{isoa}) and
(iv) to obtain that
$$
-c\la g^{-1}A,g^{-1}A\ra = \tr{S(\ad(g^{-1}A)|_{\ngo})^2} =
\tr{S(\ad'{A}|_{\ngo})^2} = -c\la A,A\ra,
$$
that is, $g\in\Or(\ag)$.   Thus $\tilde{f}$ is an orthogonal isomorphism and so it
determines an isometry between $S$ and $S'$, which concludes the proof of the
theorem.
\end{proof}

It follows from Corollary \ref{cormain} that to classify solsolitons up to isometry
one firstly has to classify nilsolitons and for each of these to consider all
possible abelian Lie algebras of symmetric derivations.  The following result gives
us the precise equivalence relation that must be considered between such algebras.

\begin{proposition}\label{toclas}
Let $(\ngo,\ip_1)$ be a nilsoliton and consider two solsolitons $S$ and $S'$
constructed as in {\rm Proposition \ref{const}} for abelian subalgebras
$$
\ag,\ag'\subset\Der(\ngo)\cap\sym(\ngo,\ip_1),
$$
respectively.  Then $S$ is isometric to $S'$ if and only if there exists
$h\in\Aut(\ngo)\cap\Or(\ngo,\ip_1)$ such that $\ag'=h\ag h^{-1}$.
\end{proposition}

\begin{proof}
If $S$ and $S'$ are isometric then we can argue as in the proof of Theorem \ref{uni}
to obtain that $h\in\Aut(\ngo)\cap\Or(\ngo,\ip_1)$ (see (\ref{ison})) and that
$\ag'=h\ag h^{-1}$ by (\ref{isoa}).

Conversely, if we define $g:\ag\longrightarrow\ag'$ by $gA=hAh^{-1}$, then
$$
\la gA,gA\ra' = -\tfrac{1}{c}\tr{(gA)^2} = -\tfrac{1}{c}\tr{(hAh^{-1})^2} =
-\tfrac{1}{c}\tr{A^2} = \la A,A\ra.
$$
This implies that $f:=\left[\begin{array}{cc} g&0\\0&h\end{array}\right]$ is an
isometric isomorphism between $(\sg,\ip)$ and $(\sg',\ip')$ and thus $f$ defines an
isometry between $S$ and $S'$, as was to be shown.
\end{proof}

\section{Examples of solsolitons}\label{exa}

Once we have chosen our favorite nilsoliton, it is quite easy to get examples of
solsolitons by using Proposition \ref{const}.  The results in Section
\ref{structure} even tell us that any solsoliton can be constructed in this simple
way.  We consider in this section the problem of which simply connected solvable Lie
groups of dimension $\leq 4$ admit a solsoliton.

\subsection{Dimension $3$}

It is well known that for any $3$-dimensional real solvable Lie algebra which is not
nilpotent there exists a basis $\{ A,X_1,X_2\}$ such that
\begin{equation}\label{par1}
\begin{array}{l}
\rg_{\alpha}: \quad [A,X_1]=X_1, \quad [A,X_2]=\alpha X_2, \qquad -1\leq\alpha\leq 1, \\ \\
\sg_{\beta}: \quad [A,X_1]=X_2, \quad [A,X_2]= -X_1+\beta X_2, \qquad 0\leq\beta\leq
2.
\end{array}
\end{equation}
The constraints on the parameters $\alpha$ and $\beta$ guarantee that these Lie
algebras are in addition pairwise non-isomorphic (see \cite{Jcb}).  It follows from
Proposition \ref{const} that the inner product $\ip$ for which $\{ A,X_1,X_2\}$ is
orthonormal is a solsoliton on $\rg_{\alpha}$ for all $\alpha$.  Recall that in all
cases we have $\sg=\ag\oplus\ngo$ for $\ag=\RR A$ and $\ngo=\RR X_1+\RR X_2$, $\ngo$
abelian. For $\sg_{\beta}$, the matrix of $\ad{A}|_{\ngo}$ relative to $\{
X_1,X_2\}$ equals $\left[\begin{smallmatrix} 0&-1\\
1&\beta\end{smallmatrix}\right]$, which has eigenvalues $\tfrac{\beta}{2}\pm
i\sqrt{1-\tfrac{\beta^2}{4}}$, and thus for any $0\leq\beta<2$, there is a basis $\{
Y_1,Y_2\}$ of $\ngo$ with respect to which $\ad{A}|_{\ngo}$ has the normal matrix
$$
\left[\begin{smallmatrix} \tfrac{\beta}{2}&-\sqrt{1-\tfrac{\beta^2}{4}}\\
\sqrt{1-\tfrac{\beta^2}{4}}&\tfrac{\beta}{2}\end{smallmatrix}\right].
$$
This implies that the inner product $\ip_1$ given by $\la
A,A\ra_1=\tr{S(\ad{A})^2}=\tfrac{\beta^2}{2}$ and with $\{ Y_1,Y_2\}$ as an
orthonormal basis for $\ngo$ is a solsoliton for $0<\beta<2$ (see Theorem
\ref{main}). But hence no new example appears, since by (\ref{sympart}) the
solsolitons $(\sg_{\beta},\ip_1)$ are either isometric to the hyperbolic space
$H^3=(\rg_1,\ip)$ for $0<\beta<2$ or to the euclidean space $\RR^3$ for $\beta=0$.

The only remaining case is $\sg_2$, for which we have that $\ad{A}|_{\ngo}$ is not
diagonalizable over $\CC$ and so it can not be normal with respect to any inner
product on $\ngo$.  We conclude from Theorem \ref{main} that $\sg_2$ is the only
$3$-dimensional solvable Lie group which does not admit a solsoliton.

{\small
\begin{table}
$$
\begin{array}{ccccccc}
\hline\hline
&&&&&& \\
 & \ngo & \ad{A}|_{\ngo} & constraints  & unimodular & solsoliton & Einstein  \\ \\
\hline
&&&&&&\\
\rg_3 & \RR^2 & \left[\begin{smallmatrix} 1&1\\ &1\end{smallmatrix}\right] & - & - & - & - \\ \\

\rg_{3,\lambda} & \RR^2 & \left[\begin{smallmatrix} 1&\\ &\lambda\end{smallmatrix}\right] & -1\leq\lambda\leq 1 & \lambda=-1 & \checkmark & \lambda=1 \\ \\

\rg_{3,\lambda}' & \RR^2 & \left[\begin{smallmatrix} \lambda&1\\ -1&\lambda\end{smallmatrix}\right] & 0\leq\lambda & \lambda=0 & \checkmark & \checkmark\\ \\
\hline \hline \\
\end{array}
$$
\caption{Classification of $3$-dimensional solvable Lie algebras admitting
solsolitons}\label{s3}
\end{table}}

There are many ways to parametrize $3$-dimensional solvable Lie algebras over $\RR$
other than (\ref{par1}).  It is in fact more convenient for our purpose to use the
description given in \cite[Theorem 1.1]{AndBrbDttOvn}, which we show in Table
\ref{s3} together with the information about the existence of solsolitons.  In this
case, if $\{ X_1,X_2\}$ is the basis of $\ngo=\RR^2$ that we use to write
$\ad{A}|_{\ngo}$, then the inner product $\ip$ making of $\{ A,X_1,X_2\}$ an
orthonormal basis is always a solsoliton.

We note that for any $0\leq\lambda$, $\rg_{3,\lambda}'$ is isomorphic to
$\sg_{\beta}$ for $\beta=\tfrac{2\lambda}{\sqrt{\lambda^2+1}}$ and $\rg_3$ is
isomorphic to $\sg_2$.  A third description can be found in \cite{Mln}.

\subsection{Dimension $4$}

It is enough to consider $4$-dimensional real solvable Lie algebras which are not
either nilpotent or a direct sum of two Lie algebras.  If $\sg=\ag\oplus\ngo$ with
$\ngo$ the nilradical of $\sg$, then the only one with $\dim{\ag}=2$ is
$\affg(\CC)$, which is defined for a basis $\{ A_1,A_2\}$ of $\ag$ by
$$
\ad{A_1}|_{\ngo}=\left[\begin{smallmatrix} 1&0\\ 0&1\end{smallmatrix}\right], \quad
\ad{A_2}|_{\ngo}=\left[\begin{smallmatrix} 0&-1\\ 1&0\end{smallmatrix}\right].
$$
It follows from Theorem \ref{main} that $\affg(\CC)$ admits a solsoliton, which is
isometric to $H^3\times\RR$ by (\ref{sympart}).

Any other has $\dim{\ag}=1$, say $\ag=\RR A$, and let $\{ X_1,X_2,X_3\}$ be a basis
of $\ngo$, which will be assumed to satisfy $[X_1,X_2]=X_3$ when $\ngo$ is not
abelian (i.e. $\ngo$ isomorphic to the $3$-dimensional Heisenberg Lie algebra
$\hg_3$).  We have used the classification given in \cite[Theorem 1.5]{AndBrbDttOvn}
to summarize all the relevant information in Table \ref{n3}, which has been obtained
as in the $3$-dimensional case above by a direct application of Proposition
\ref{const} and Theorem \ref{main}.  The map $\ad{A}|_{\ngo}$ is always written in
terms of the basis $\{ X_1,X_2,X_3\}$ and the inner product $\ip$ for which $\{
A,X_1,X_2,X_3\}$ is an orthonormal basis is always a solsoliton, when there is one.

{\small
\begin{table}
$$
\begin{array}{ccccccc}
\hline\hline
&&&&&& \\
 & \ngo & \ad{A}|_{\ngo} & constraints  & unimodular & solsoliton & Einstein  \\ \\
\hline
&&&&&&\\
\rg_4 & \RR^3 & \left[\begin{smallmatrix} 1&1&\\ &1&1\\ &&1\end{smallmatrix}\right] & - & - & - & - \\ \\

\rg_{4,\lambda} & \RR^3 & \left[\begin{smallmatrix} 1&&\\ &\lambda&1\\ &&\lambda\end{smallmatrix}\right] & -\infty<\lambda<\infty & \lambda=-\unm & - & - \\ \\

\rg_{4,\mu,\lambda} & \RR^3 & \left[\begin{smallmatrix} 1&&\\ &\mu&\\ &&\lambda\end{smallmatrix}\right] & \begin{array}{l} -1<\mu\leq\lambda\leq 1; \\ -1=\mu\leq\lambda<0\end{array} & \mu=-1-\lambda & \checkmark & \mu=\lambda=1\\ \\

\rg_{4,\mu,\lambda}' & \RR^3 & \left[\begin{smallmatrix} \mu&&\\ &\lambda&1\\ &-1&\lambda\end{smallmatrix}\right] & 0<\mu & \mu=-2\lambda & \checkmark & \mu=\lambda \\ \\

\sg_4 & \hg_3 & \left[\begin{smallmatrix} 1&&\\ &-1&\\ &&0\end{smallmatrix}\right] & - & \checkmark & \checkmark & - \\ \\

\sg_{4,\lambda} & \hg_3 & \left[\begin{smallmatrix} \lambda&&\\ &1-\lambda&\\ &&1\end{smallmatrix}\right] & \unm\leq\lambda & - & \checkmark & \lambda=\unm \\ \\

\sg_{4,\lambda}' & \hg_3 & \left[\begin{smallmatrix} \lambda&1&\\ -1&\lambda&\\ &&2\lambda\end{smallmatrix}\right] & 0\leq\lambda & \lambda=0 & \lambda\ne 0 & \lambda\ne 0 \\ \\

\hg_4 & \hg_3 & \left[\begin{smallmatrix} 1&1&\\ &1&\\ &&2\end{smallmatrix}\right] & - & - & - & - \\ \\
\hline \hline \\
\end{array}
$$
\caption{Classification of $4$-dimensional solvable Lie algebras with a
$3$-dimensional nilradical admitting a solsoliton}\label{n3}
\end{table}}

By using (\ref{sympart}), we get that $\rg_{4,\mu,\lambda}'$, $\lambda\ne 0$, is
isometric to $\rg_{4,\lambda/\mu,\lambda/\mu}$ and that $\sg_{4,\lambda}'$ is
isometric to $\sg_{4,1/2}$ for any $\lambda>0$.

One can see in Table \ref{n3} that $\rg_{4,-1/2}$ is the only unimodular solvable
Lie algebra of dimension $4$ which does not admit a solsoliton.  It follows however
from the results in \cite{Hng} that $\rg_{4,-1/2}$ does not either admit a {\it
lattice} (i.e. cocompact discrete subgroup).  Thus the universal cover of any
$4$-dimensional compact solvmanifold $S/\Gamma$ does admit a solsoliton, a result
which has already been proved in \cite{IsnJckLu}.  This is essentially due to the
fact from algebraic number theory that any $A\in\Sl_3(\ZZ)$ with positive real
eigenvalues and at least one of them different from $1$ is necessarily
diagonalizable over $\CC$.  The next example shows that this is no longer true for
compact solvmanifolds of dimension $\geq 5$.

\begin{example}
Let $\sg=\RR A\oplus\ngo$ be the solvable Lie algebra defined by: $\ngo$ is abelian,
$\dim{\ngo}=4$ and
$$
\ad{A}|_{\ngo}= \left[\begin{smallmatrix} 0&1&&\\ 0&0&&\\ &&\ln{\lambda}&\\
&&&-\ln{\lambda}\end{smallmatrix}\right], \qquad\lambda=\tfrac{3+\sqrt{5}}{2}.
$$
Thus there exists $\sigma\in\Gl_4(\RR)$ such that
$$
\sigma e^{\ad{A}|_{\ngo}}\sigma^{-1} =  \left[\begin{smallmatrix} 1&1&&\\ 0&1&&\\
&&2&1\\ &&1&1\end{smallmatrix}\right]\in\Sl_4(\ZZ),
$$
and it is therefore easy to check that $\Gamma:=\ZZ\ltimes\sigma^{-1}\ZZ^4$ is a
lattice of the solvable Lie group $S$ with Lie algebra $\sg$.  Recall that
$S=\RR\ltimes\RR^4$ with multiplication given by
$$
(t,X).(s,Y) = (t+s, X+e^{t\ad{A}|_{\ngo}}Y), \qquad t,s\in\RR, \quad X,Y\in\RR^4.
$$
However, since $\ad{A}|_{\ngo}$ is not diagonalizable over $\CC$, we conclude from
Theorem \ref{main} that $S$ can never admit a solsoliton.
\end{example}


\begin{thebibliography}{MMMM}

\bibitem[A]{Alk} {\sc D. Alekseevskii}, Conjugacy of polar factorizations of Lie groups, {\it Mat. Sb.} {\bf 84} (1971), 14-26; {\it English translation}: {\it Math. USSR-Sb.} {\bf 13} (1971), 12-24.

\bibitem[AK]{AlkKml} {\sc D. Alekseevskii, B. Kimel'fel'd}, Structure of homogeneous Riemannian spaces with zero Ricci curvature, {\it Funktional Anal. i Prilozen} {\bf 9} (1975), 5-11 (English translation: {\it Functional Anal. Appl.} {\bf 9} (1975), 97-102.

\bibitem[ABDO]{AndBrbDttOvn} {\sc A. Andrada, L. Barberis, I. Dotti, G. Ovando}, Product structures on four dimensional solvable Lie algebras, {\it Homol. Homot. Appl.} {\bf 7} (2005), 9-37.

\bibitem[AT]{AylTir} {\sc V. Ayala, J. Tirao}, Linear control systems on Lie groups and controllability, {\it Proc. Symp. Pure Math.} {\bf 64} (1999), 47-64.

\bibitem[AW]{AznWls} {\sc R. Azencott, E. Wilson}, Homogeneous manifolds with negative curvature I,
{\it Trans. Amer. Math. Soc.} {\bf 215} (1976), 323-362.

\bibitem[BD]{BrdDnl} {\sc P. Baird, L. Danielo}, Three-dimensional Ricci solitons which project to surfaces, {\it J. reine angew. Math.} {\bf 608} (2007), 65-91.

\bibitem[B]{Bss} {\sc A. Besse}, Einstein manifolds, {\it Ergeb. Math.} {\bf 10} (1987), Springer-Verlag,
Berlin-Heidelberg.

\bibitem[BWZ]{BhmWngZll} {\sc C. B$\ddot{{\rm o}}$hm, M.Y. Wang, W. Ziller}, A variational approach for
compact homogeneous Einstein manifolds, {\it Geom. Funct. Anal.} {\bf 14} (2004),
681-733.

\bibitem[C]{libro}  {\sc B. Chow, S.-C. Chu, D. Glickenstein, C. Guenther, J. Isenberg,
T, Ivey, D. Knopf, P. Lu, F. Luo, L. Ni}, The Ricci flow: Techniques and
Applications, Part I: Geometric Aspects, {\it AMS Math. Surv. Mon.} {\bf 135}
(2007), Amer. Math. Soc., Providence.

\bibitem[EJ]{EbrJbl} {\sc P. Eberlein, M. Jablonski}, Closed orbits of semisimple group actions and the real Hilbert-Mumford function, {\it Contemp. Math.} {\bf 491} (2009), 283-321.

\bibitem[GW]{GrdWls} {\sc C. Gordon, E. Wilson}, Isometry groups of Riemannian solvmanifolds,
{\it Trans. Amer. Math. Soc.} {\bf 307} (1988), 245-269.

\bibitem[H]{Hbr} {\sc J. Heber}, Noncompact homogeneous Einstein spaces, {\it Invent. math}. {\bf 133} (1998), 279-352.

\bibitem[HSS]{HnzSchStt} {\sc P. Heinzner, G. Schwarz, H. St$\ddot{{\rm o}}$etzel}, Stratifications with respect to actions of real reductive groups, {\it Compositio Math.}  {\bf 144} (2008), 163-185.

\bibitem[Hn]{Hng} {\sc H. Huang}, Lattices and harmonic analysis on some $2$-step solvable Lie groups, {\it J. Lie Theory} {\bf 13} (2003), 77-89.

\bibitem[IJL]{IsnJckLu} {\sc J. Isenberg, M. Jackson, Peng Lu}, Ricci flow on locally homogenous closed $4$-manifolds, {\it Comm. Anal. Geom.} {\bf 14} (2006), 345-386.

\bibitem[I]{Ivy} {\sc T. Ivey}, Ricci solitons on compact three-manifolds, {\it Diff. Geom. Appl.} {\bf 3} (1993), 301-307.

\bibitem[Jc]{Jcb} {\sc N. Jacobson}, Lie algebras, {\it Interscience Publishers}, New York (1962).

\bibitem[KN]{KmpNss} {\sc G. Kempf, L. Ness}, The length of vectors in representation spaces,
{\it Lect. Notes in Math.} {\bf 732} (1979), 233-243.

\bibitem[K]{Krw1} {\sc F. Kirwan}, Cohomology of quotients in symplectic and algebraic
geometry, {\it Mathematical Notes} {\bf 31} (1984), Princeton Univ. Press,
Princeton.

\bibitem[L1]{soliton} {\sc J. Lauret}, Ricci soliton homogeneous nilmanifolds,
{\it Math. Annalen} \textbf{319} (2001), 715-733.

\bibitem[L2]{minimal} \bysame, A canonical compatible metric for geometric structures on
nilmanifolds, {\it Ann. Global Anal. Geom.} {\bf 30} (2006), 107-138.

\bibitem[L3]{standard}  \bysame, Einstein solvmanifolds are standard, {\it Ann. of Math.}, in press,
{\it arXiv:} math.DG/0703472.

\bibitem[L4]{cruzchica}  \bysame, Einstein solvmanifolds and nilsolitons, {\it Contemp. Math.} {\bf 491} (2009), 1-35.

\bibitem[LW]{einsteinsolv}  {\sc J. Lauret, C.E. Will}, Einstein solvmanifolds: existence
and non-existence questions, preprint 2006, {\it arXiv:} math.DG/0602502.

\bibitem[Lt]{Ltt07}  {\sc J. Lott}, Dimensional reduction and the long-time behavior of Ricci flow, {\it Comm. Math. Helv.}, in press, {\it
 arXiv:} math.DG/0711.4063.

\bibitem[M]{Mrn} {\sc A. Marian}, On the real moment map, {\it Math. Res. Lett.} {\bf 8} (2001), 779-788.

\bibitem[Ml]{Mln} {\sc J. Milnor}, Curvature of Left-invariant Metrics on Lie Groups, {\it Adv. Math.} {\bf 21}(1976), 293-329.

\bibitem[MFK]{Mmf} {\sc D. Mumford, J. Fogarty, F. Kirwan}, Geometric invariant theory, Third Edition,
Springer Verlag (1994).

\bibitem[N]{Nbr}  {\sc A. Naber}, Noncompact shrinking four solitons with nonnegative curvature, preprint 2008,
{\it arXiv:} math.DG/0710.5579.

\bibitem[Ns]{Nss} {\sc L. Ness}, A stratification of the null cone via the momentum map, {\it Amer. J. Math.} {\bf 106} (1984), 1281-1329 (with an appendix by D. Mumford).

\bibitem[P]{Prl} {\sc G. Perelman}, The entropy formula for the Ricci flow and its geometric applications, preprint 2002, {\it arXiv:} math.DG/0211159.

\bibitem[PW]{PtrWyl}  {\sc P. Petersen, W. Wylie}, On gradient Ricci solitons with symmetry,  {\it Proc. Amer. Math. Soc.} {\bf 137} (2009), 2085-2092.

\bibitem[RS]{RchSld} {\sc R.W. Richardson, P.J. Slodowy}, Minimum vectors for real reductive algebraic groups,
{\it J. London Math. Soc. (2)} {\bf 42} (1990), 409-429.

\bibitem[W]{Wyl} {\sc W. Wylie}, Complete shrinking Ricci solitons have finite fundamental group, {\it Proc. Amer. Math. Soc.} {\bf 136} (2008), 1803-1806.

\end{thebibliography}
\end{document}